\documentclass[11pt, a4paper]{amsart}
\usepackage[latin1]{inputenc}
\usepackage[T1]{fontenc}
\usepackage[english]{babel}
\usepackage{amssymb, amsmath, amsthm}
\usepackage{hyperref}
\usepackage{geometry}
\usepackage{fancyhdr}
\usepackage{relsize}
\usepackage{upgreek}
\usepackage{verbatim}
\usepackage{enumitem}
\usepackage{calrsfs}
\usepackage{todonotes}
\usepackage{tikz}
\usetikzlibrary{matrix,arrows,calc}
\usepackage{package}
\newcommand{\QQ}{\mathcal{Q}}
\DeclareMathOperator{\dist}{dist}
\title{Rationality of Rigid Quiver Grassmannians}
\author{Hans Franzen}
\address{Mathematisches Institut der Universit\"at Bonn\\Endenicher Allee 60\\53115 Bonn (Germany)}
\email{franzen@math.uni-bonn.de}
\date{}
\begin{document}
	
	\begin{abstract}
		We show that any quiver Grassmannian associated with a rigid representation of a quiver is a rational variety using torus localization techniques.
	\end{abstract}
	
	\maketitle

	\section{Introduction}
	
	A quiver Grassmannian $\Gr_e(M)$ is a projective scheme associated with a fixed representation $M$; its points parametrize subrepresentations of $M$ of a given dimension vector $e$. They have been introduced by Schofield in \cite{Schofield:92}. It was already remarked by Schofield that quiver Grassmannians are, in general, neither irreducible, nor reduced; in particular they are not smooth in general. In fact, every projective variety can be obtained as a quiver Grassmannian as was shown independently by Hille \cite{Hille:15}, Huisgen-Zimmermann \cite{Huisgen-Zimmermann:07}, and Reineke \cite{Reineke:13}. This shows that we cannot expect any special properties when studying arbitrary quiver Grassmannians. 
	
	However, quiver Grassmannians associated with rigid representations have gained importance in the theory of cluster algebras due to work of Caldero and Chapoton who show that Euler characteristics of these quiver Grassmannians appear in cluster variables \cite{CC:06}. Hence the positivity conjecture of Fomin--Zelevinsky \cite{FZ:07} translates into positivity of the Euler characteristic of rigid quiver Grassmannians. This conjecture was settled independently by Nakajima \cite{Nakajima:11} and Qin \cite{Qin:12}. Both authors show, using different methods, that the (Betti) cohomology of $\Gr_e(M)$ is concentrated in even degrees provided that $M$ is rigid. This result was generalized in \cite{CEFR:18} where it is shown that the class of the diagonal inside the Chow ring of $\Gr_e(M) \times \Gr_e(M)$ is contained in the image of the exterior product map. Formal consequences of this are that the cohomology of $\Gr_e(M)$ is torsion-free and vanishes in odd degrees and that the cycle map is an isomorphism. 
	
	This might be an indicator for the existence of a cellular decomposition. Indeed there are several classes of examples of rigid quiver Grassmannian which admit a cellular decomposition, such as quiver Grassmannians attached to indecomposable representations of a quiver of finite or affine type \cite{CEFR:18} (special cases of this were shown before in \cite{Cerulli_Irelli:11, CE:11, CFFFR:17, LW:15:1, LW:15:2}). In \cite{RW:18} Rupel and Weist show that quiver Grassmannians of exceptional representations of a generalized Kronecker quiver admit a cellular decomposition. Although this is good evidence that quiver Grassmannians of rigid representations might always possess a cellular decomposition, it is not known, as of now, if they actually do.
	
	A necessary condition for the existence of a cellular decomposition is rationality.
	The main objective of the present paper is to show
	\begin{thm} \label{t:main}
		Let $M$ be a rigid representation of a quiver $Q$ and let $e \leq \dimvect M$ be a dimension vector. Then $\Gr_e(M)$ is rational.
	\end{thm}
	This narrows down the class of possible counterexamples to the conjectured existence of a cellular decomposition.
	
	The proof of Theorem \ref{t:main} relies on two (iterated) torus localizations. First, we define a torus action on $\Gr_e(M)$ to reduce to the case of an exceptional representation. This is Lemma \ref{l:red1}. A second torus action reduces rationality of a quiver Grassmannian of an exceptional representation of $Q$ to rationality of an exceptional representation of the universal abelian covering quiver. This is done in Lemma \ref{l:red2}. Applying this construction a finite number of times we arrive at an exceptional, in particular rigid, representation of a tree. This case is solved directly in Lemma \ref{l:treeRational}.
	
	\begin{ack*}
		I would like to thank Giovanni Cerulli Irelli, Francesco Esposito, and Thorsten Weist for helpful discussions on the paper. I am grateful to Markus Reineke for suggesting a simpler proof of Lemma \ref{l:rigid}.
	\end{ack*}
	
	\section{Generalities on Quiver Grasmmannians}
	
	Let $Q$ be a quiver, i.e.\ an oriented graph. We denote its set of vertices by $Q_0$ and its set of arrows by $Q_1$. For an arrow $\alpha \in Q_1$ let $s(\alpha)$ and $t(\alpha)$ its starting vertex and terminal vertex, respectively. Let $k$ be an algebraically closed field. Let $M$ be a finite-dimensional representation of $Q$ over $k$. All the representations considered in this article are assumed to be finite-dimensional. For elementary facts on representations of quivers, we refer to \cite[Chap.\ II]{ASS:06}. Recall that a subrepresentation of $M$ is a collection $(U_i)_{i \in Q_0}$ of subspaces $U_i \sub M_i$ which satisfy $M_\alpha(U_{s(\alpha)}) \sub U_{t(\alpha)}$ for every $\alpha \in Q_1$. Fix a sub--dimension vector $e \in \smash{\Z_{\geq 0}^{Q_0}}$ of $d := \dimvect M$, i.e.\ $e_i \leq d_i$ for every $i \in Q_0$. The set 
	$$
		\Big\{ (U_i)_i \in \prod_{i \in Q_0} \Gr_{e_i}(M_i) \mid M_\alpha(U_{s(\alpha)}) \sub U_{t(\alpha)} \text{ (all $i \in Q_0$)} \Big\}
	$$
	of all subrepresentations of $M$ is a Zariski-closed subset of (the set of $k$-valued points of) the product $\prod_{i \in Q_0} \Gr_{e_i}(M_i)$ of Grassmannians. There is a natural way to make this closed subset into a scheme. For this let $\UU_j$ be the pull-back along the projection $\prod_i \Gr_{e_i}(M_i) \to \Gr_{e_j}(M_j)$ of the universal rank $e_j$-bundle on $\Gr_{e_j}(M_j)$; it is a subbundle of the trivial bundle with fiber $M_j$. Let $\QQ_j$ be the pull-back along the same map of the universal quotient bundle on $\Gr_{e_j}(M_j)$. Let $s_M$ be the global section of the bundle $\bigoplus_{\alpha \in Q_1} \smash{\UU_{s(\alpha)}^\vee} \otimes \QQ_{t(\alpha)}$ which, over a (closed) point $(U_i)_i \in \prod_i \Gr_{e_i}(M_i)$, is given by the sum of the linear maps 
	$$
		U_i \to M_i \xto{}{M_\alpha} M_j \to M_j/U_j.
	$$ 
	The set of $k$-valued points of the zero scheme $Z(s_M)$ of this section is precisely the set of subrepresentations $U$ of $M$. 
	
	\begin{defn*}
		The quiver Grassmannian $\Gr_e(M)$ of subrepresentations of $M$ of dimension vector $e$ is defined as the zero scheme $Z(s_M)$.
	\end{defn*}
	
	A representation $M$ of $Q$ is called rigid if $\Ext^1(M,M) = 0$. For a rigid representation $M$ the quiver Grassmannian $\Gr_e(M)$ is particularly nice:

	\begin{lem}
		If $M$ is a rigid representation then the scheme $\Gr_e(M)$ is a smooth projective variety of dimension $\langle e,d-e \rangle$; in particular it is reduced and irreducible.
	\end{lem}
	
	In the above lemma, $d$ is the dimension vector of $M$ and $\langle \blank,\blank \rangle$ is the Euler form of the quiver $Q$.
	
	\begin{proof}
		Irreducibility was shown in \cite[Thm.\ 4.12]{Wolf:09}. In \cite[Cor.\ 2]{CR:08}, smoothness and the dimension formula are proved.
	\end{proof}

	\section{Covering Quivers}
	
	We introduce covering quivers of a quiver without relations. The main reference for us is Weist's paper \cite{Weist:13}. The definitions of a cover and a universal cover is taken from \cite[App.\ 2]{Zhang:91}.
	
	\begin{defn*}
		Let $Q$ be a quiver.
		A cover of $Q$ is a pair consisting of a quiver $C$ and a morphism $c = (c_0,c_1): C \to Q$ of quivers which is surjective, both on vertices and on arrows, such that for any $i \in Q_0$ and any $k \in C_0$ with $c_0(k) = i$ the map $c_1$ on arrows induces bijections 
		\begin{align*}
			\{ \beta \in C_1 \mid s(\beta) = k \} &\xto{}{1:1} \{\alpha \in Q_1 \mid s(\alpha) = i \} \\
			\{ \beta \in C_1 \mid t(\beta) = k \} &\xto{}{1:1} \{ \alpha \in Q_1 \mid t(\alpha) = i \}.
		\end{align*}
	\end{defn*}
	
	Let $\bar{Q}$ be the double quiver of $Q$. It has vertices $\bar{Q}_0 = Q_0$ and arrows $\bar{Q}_1 = \{ \alpha^{\epsilon} \mid \alpha \in Q_1,\ \epsilon \in \{\pm 1\}\}$ which are directed as follows:
	\begin{align*}
		s(\alpha^{\epsilon}) &= \begin{cases} s(\alpha) & \text{if } \epsilon = +1 \\ t(\alpha) & \text{if } \epsilon = -1 \end{cases}
		& t(\alpha^{\epsilon}) &= \begin{cases} t(\alpha) & \text{if } \epsilon = +1 \\ s(\alpha) & \text{if } \epsilon = -1 \end{cases}
	\end{align*}
	An unoriented path of $Q$ is by definition a path in $\bar{Q}$. The set of all unoriented paths is denoted $\bar{Q}_*$. For standard notions for paths of quivers we refer to \cite[II.1]{ASS:06}. If an unoriented path of $Q$ is not a lazy path of $\bar{Q}$ then it is completely determined by a non-empty sequence $p = \alpha_r^{\epsilon_r}\ldots\alpha_1^{\epsilon_1}$ such that $s(\alpha_{\nu+1}^{\epsilon_{\nu+1}}) = t(\alpha_\nu^{\epsilon_\nu})$ for every $\nu=1,\ldots,n-1$. The unoriented path $p$ is called an unoriented cycle if $s(p) = t(p)$. An unoriented path $p$ is called reduced if $l(p) > 0$ and if it does not contain an expression of the form $\alpha\alpha^{-1}$ or $\alpha^{-1}\alpha$ as a subword. 
	
	A quiver $Q$ is called connected if for any two vertices $i,j \in Q_0$ there exists an unoriented path $p$ such that $s(p) = i$ and $t(p) = j$. A quiver is called tree-shaped if it is connected and if there are no reduced unoriented cycles in $Q$.
	
	\begin{defn*}
		A universal cover of a quiver $Q$ is a cover $C$ which is tree-shaped.
	\end{defn*}

	\begin{lem} \label{l:univ}
		Let $Q$ be a quiver.
		\begin{enumerate}
			\item Let $c: C \to Q$ be a universal cover and let $d: D \to Q$ be any cover. Then there exists a morphism of quivers $p: C \to D$ such that $d \circ p = c$.
			\item If $Q$ is connected then it admits a universal cover.
		\end{enumerate}
	\end{lem}
	
	The second statement was proved in \cite[Prop.\ 3.12]{Weist:13} where it is shown that every connected component of the universal covering quiver is a universal cover.
	The first statement is well known. As we could not find a formal proof we give one here.
	
	\begin{proof}[Proof of (1)]
		For a connected quiver define $\dist(i,j) := \min \{ l(p) \mid p \in \bar{Q}_* \text{ with } s(p) = i \text{ and } t(p) = j \}$ for vertices $i,j \in Q_0$ and $\dist(\alpha,j) = \min\{ \dist(s(\alpha),j), \dist(t(\alpha),j)\}$.
		As a universal cover is by definition connected, the quiver $Q$ must be connected as well. Fix a vertex $i_0 \in Q_0$. Fix $k_0 \in C_0$ and $l_0 \in D_0$ such that $c_0(k_0) = d(l_0) = i_0$. We define $p$ inductively over the distance. Define $p_0(k_0) = l_0$. Now let $d > 0$ and assume inductively that $p_0$ has been defined on all vertices of $C$ within a distance of $d-1$ to $k_0$ and $p_1$ on all arrows within a distance of $d-2$. Let $\beta_1,\ldots,\beta_n$ be those arrows of $C$ such that $\dist(\beta_\nu,k_0) = d-1$. Then $\beta_\nu$ is adjacent to $k_\nu \in C_1$ such that $\dist(k_\nu,k_0) = d-1$. Let $k_\nu'$ be the other endpoint of $\beta_\nu$. Then $\dist(k_\nu',k_0) = d$ and $k_\nu' \neq k_\mu'$ for $\nu \neq \mu$ (otherwise there would be unoriented cycles). Let $\gamma_\nu$ be the unique arrow adjacent to $p_0(k_\nu)$ such that $c_1(\beta_\nu) = d_1(\gamma_\nu)$. Define $p_1(\beta_\nu) = \gamma_\nu$ and $p_0(k_\nu') = l_\nu'$ where $l_\nu'$ is the other endpoint of $\gamma_\nu$. Proceeding in this way we reach every arrow and every vertex of $C$ as it is connected.
	\end{proof}
	
	Let $c: C \to Q$ be an arbitrary cover. It induces a functor $c: \operatorname{Rep}(C) \to \operatorname{Rep}(Q)$ as follows: for a representation $N$ of $C$ define $c(N)$ as the representation of $Q$ which on vertices is given by 
	$c(N)_i = \bigoplus_{c(k) = i} N_k$ and for an arrow $\alpha: i \to j$ the linear map $c(N)_\alpha: c(N)_i \to c(N)_j$ is defined by
	$$
		c(N)_\alpha\Big(\sum_{c_0(k)=i} v_{i,k}\Big) = \sum_{c_0(k)=i} N_{\beta_k}(v_{i,k})
	$$
	where for every $k$ such that $c_0(k) = i$, the arrow $\beta_k$ is the unique arrow of $C$ such that $s(\beta_k) = k$ and $c_1(\beta_k) = \alpha$. For a representation $M$ of $Q$, a representation $N$ of $C$ for which $c(N) = M$ is called a lift of $M$ to $C$.
	
	Let $M$ be an exceptional representation of $Q$, that means $M$ is rigid and indecomposable. Ringel shows in \cite{Ringel:98} that $M$ is a tree module. As tree modules lift to a universal cover, we obtain:
	
	\begin{lem} \label{l:ringel}
		Let $c: C \to Q$ be a universal cover of $Q$. Any exceptional representation of $Q$ possesses a lift to $C$.
	\end{lem}

	We need an abelian approximation to the universal cover. For this we use the universal abelian covering quiver which was introduced by Weist in \cite{Weist:13}.
	The universal abelian covering quiver $\hat{Q}$ of $Q$ is defined as
	\begin{align*}
		\hat{Q}_0 &= Q_0 \times \Z^{Q_1} & \hat{Q}_1 &= Q_1 \times \Z^{Q_1}
	\end{align*}
	where, for a pair $(\alpha,\chi) \in \hat{Q}_1$ we define $s(\alpha,\chi) = (s(\alpha),\chi)$ and $t(\alpha,\chi) = (t(\alpha),\chi+\alpha)$. The morphism $\smash{\hat{c}}: \smash{\hat{Q}} \to Q$ of quivers defined by $\smash{\hat{c}_0}(i,\chi) = i$ and $\smash{\hat{c}_1}(\alpha,\chi) = \alpha$ is a cover. A representation $M$ of $Q$ has a lift $\smash{\hat{M}}$ to $\smash{\hat{Q}}$ if and only if there exist direct sum decompositions $M_i = \smash{\bigoplus_\chi} \smash{\hat{M}_{i,\chi}}$ in such a way that $M_\alpha(\smash{\hat{M}_{s(\alpha),\chi}}) \sub \smash{\hat{M}_{t(\alpha),\chi+\alpha}}$.
	
	We want to show that rigidity and exceptionality are preserved for lifts. Let $\smash{\hat{M}}$ be a representation of $\smash{\hat{Q}}$ and let $M = \smash{\hat{c}(\hat{M})}$. Let $\xi \in X(T)$ and define $S_{\xi}(\smash{\hat{M}})$ as the representation with $S_\xi(\smash{\hat{M}})_{i,\chi} = \smash{\hat{M}_{i,\chi+\xi}}$ and $S_\xi(\smash{\hat{M}})_{\alpha,\chi} = \smash{\hat{M}_{\alpha,\chi+\xi}}$. This defines an endofunctor on the category of representations of $\smash{\hat{Q}}$.
	
	\begin{lem} \label{l:rigidCover}
		Let $\smash{\hat{M}}$ and $\smash{\hat{N}}$ be two representations of $\smash{\hat{Q}}$ and let $M = \smash{\hat{c}(\hat{M})}$ and $N = \smash{\hat{c}(\hat{N})}$. Then
		\begin{align*}
			\Hom_Q(M,N) &\cong \bigoplus_{\xi \in X(T)} \Hom_{\hat{Q}}(S_\xi(\hat{M}),\hat{N}) &
			\Ext_Q(M,N) &\cong \bigoplus_{\xi \in X(T)} \Ext_{\hat{Q}}(S_\xi(\hat{M}),\hat{N})
		\end{align*}
		In particular, if $M$ is rigid or exceptional, then so is $\smash{\hat{M}}$.
	\end{lem}
	
	\begin{proof}
		Consider the exact sequence
		$$
			0 \to \Hom_Q(M,N) \to \bigoplus_i \Hom(M_i,N_i) \xto{}{\phi} \bigoplus_\alpha \Hom(M_{s(\alpha)},M_{t(\alpha)}) \to \Ext_Q(M,N) \to 0.
		$$
		which arises by applying the functor $\Hom(\blank,N)$ to the standard projective resolution of $M$. The map $\phi$ is given by $\phi(\sum_i f_i) = \sum_\alpha N_\alpha f_{s(\alpha)} - f_{t(\alpha)} M_\alpha$. The second term of the exact sequence decomposes as
		\begin{align*}
			\bigoplus_i \Hom(M_i,N_i) &= \bigoplus_i \bigoplus_{\chi,\eta} \Hom(\hat{M}_{i,\chi},\hat{N}_{i,\eta}) = \bigoplus_\xi \bigoplus_{(i,\chi)} \Hom(\hat{M}_{i,\chi+\xi},\hat{N}_{i,\chi}) \\ 
			&= \bigoplus_\xi \bigoplus_{(i,\chi)} \Hom(S_\xi(\hat{M})_{i,\chi},\hat{N}_{i,\chi}).
		\end{align*}
		In the same vein we identify
		$$
			\bigoplus_\alpha \Hom(M_{s(\alpha)},M_{t(\alpha)}) = \bigoplus_\xi \bigoplus_{(\alpha,\chi)} \Hom(S_\xi(\hat{M})_{i,\chi},\hat{N}_{i,\chi+\alpha}).
		$$
		Using these identifications the map $\phi$ sends $\sum_\xi \sum_{(i,\chi)} f_{\xi,(i,\chi)}$ to
		$$
			\sum_\xi \sum_{(\alpha,\chi)} \hat{N}_{\alpha,\chi} f_{\xi,(s(\alpha),\chi)} - f_{\xi,(t(\alpha),\chi+\alpha)} \hat{M}_{\alpha,\chi+\xi}
		$$
		which shows the desired isomorphisms of the $\Hom$ and $\Ext$ groups. The consequence that rigidity and exceptionality are preserved for lifts is clear.
	\end{proof}	 

	In contrast to a universal cover the covers $\smash{\hat{Q}}$ may have reduced unoriented cycles. However, if we denote by $l_Q$ the minimal length of all reduced unoriented cycles in the underlying graph of $Q$, we see that this number gets bigger when passing to the universal abelian covering quiver:
	
	\begin{lem} \label{l:cycles}
		Let $Q$ be a quiver and let $\hat{Q}$ its universal abelian covering quiver. Then $l_Q < l_{\hat{Q}}$.
	\end{lem}
	
	This was proved implicitly in \cite[Prop.\ 3.12]{Weist:13}. We include a proof for the convenience of the reader.
	
	\begin{proof}
		Consider a reduced unoriented cycle $\hat{p} = (\alpha_r,\chi_r)^{\epsilon_r}\ldots(\alpha_1,\chi_1)^{\epsilon_1}$ of $\smash{\hat{Q}}$ which is of minimal length $r = \smash{l_{\hat{Q}}}$. Note that
		\begin{align*}
			s((\alpha,\chi)^{\epsilon}) &= \begin{cases} (s(\alpha),\chi) & \text{if } \epsilon = +1 \\ (t(\alpha),\chi+\alpha) & \text{if } \epsilon = -1 \end{cases} &
			t((\alpha,\chi)^{\epsilon}) &= \begin{cases} (t(\alpha),\chi+\alpha) & \text{if } \epsilon = +1 \\ (s(\alpha),\chi) & \text{if } \epsilon = -1 \end{cases} \\
			&= (s(\alpha^{\epsilon}),\chi+\frac{1-\epsilon}{2}\alpha) &			
			&= (t(\alpha^{\epsilon}),\chi+\frac{1+\epsilon}{2}\alpha)
		\end{align*}
		As $\hat{p}$ is a cycle, we obtain that $p := \hat{c}(\hat{p}) = \alpha_r^{\epsilon_r}\ldots\alpha_1^{\epsilon_1}$ is a cycle in the underlying graph of $Q$ and the equations
		\begin{align*}
			\chi_{\nu+1} + \frac{1-\epsilon_{\nu+1}}{2}\alpha_{\nu+1} &= \chi_\nu + \frac{1+\epsilon_\nu}{2}\alpha_\nu \\
			\chi_{1} + \frac{1-\epsilon_{1}}{2}\alpha_{1} &= \chi_n + \frac{1+\epsilon_n}{2}\alpha_n
		\end{align*}
		force
		$$
			\epsilon_1\alpha_1 + \ldots \epsilon_r\alpha_r = 0.
		$$
		This equation implies that for every arrow $\alpha$ which occurs in $p$ its formal inverse $\alpha^{-1}$ occurs as well (and vice versa). As the sequence of $\hat{p}$ is reduced hence so is $p$, and thus $\alpha$ and $\alpha^{-1}$ cannot appear as neighbors in the cyclic order. This means that the cycle $p$ visits the same vertex twice and hence possesses a proper subcycle.
	\end{proof}
	
	To get an approximation to the universal cover we iteratively define
	$$
		\hat{Q}^{(n)} := \hat{ \hat{Q}^{(n-1)} }
	$$
	for every $n > 1$; the quiver $\smash{\hat{Q}^{(0)}}$ is defined as $Q$. Denote the covering morphism with $\smash{\hat{c}_{n,n-1}}: \smash{\hat{Q}^{(n)}} \to \smash{\hat{Q}^{(n-1)}}$. The concatenation of these covering morphism yields a cover $\smash{\hat{c}^{(n)}}: \smash{\hat{Q}^{(n)}} \to Q$. From Lemma \ref{l:univ} we deduce
	
	\begin{lem} \label{l:covermaps}
		Let $c: C \to Q$ be a universal cover of $Q$. Then for every $n \geq 0$ there exists a morphism of quivers $\smash{p^{(n)}}: C \to \smash{\hat{Q}^{(n)}}$ such that $\smash{p^{(0)}} = c$ and 
		$$
			\hat{c}_{n,n-1} \circ p^{(n)} = p^{(n-1)}.
		$$
	\end{lem}
	
	The restriction of $p^{(n)}$ to its range is a cover. This implies that $\smash{p^{(n)}}$ induces a functor $\smash{p^{(n)}}: \operatorname{Rep}(C) \to \operatorname{Rep}(\smash{\hat{Q}^{(n)}})$.
	
	\begin{lem} \label{l:lift}
		Let $M$ be a representation of $Q$ which has a lift $N$ to a universal cover $C$ of $Q$. Choose for any $n \geq 0$ morphism of quivers $\smash{p^{(n)}}$ as in Lemma \ref{l:covermaps}. Then $\smash{p^{(n)}}(N)$ is a lift of $M$ to $\smash{\hat{Q}^{(n)}}$. There exists an $n > 0$ such that the support of $\smash{p^{(n)}}(N)$ is a tree.
	\end{lem}
	
	\begin{proof}
		The assertion that $\smash{p^{(n)}}(N)$ is a lift of $M$ to $\smash{\hat{Q}}^{(n)}$ follows from $c = \hat{c}^{(n)} \circ p^{(n)}$. By Lemma \ref{l:cycles} we find a natural number $n$ such that $\smash{l_{\hat{Q}^{(n)}}} > \dim_k M$. This implies that the support of $\smash{p^{(n)}}(N)$ is a tree.
	\end{proof}

	\section{Torus Actions on Quiver Grassmannians}
	
	We consider two types of torus actions on the quiver Grasmmannian of a rigid representation and reduce the question of rationality to the components of the fixed point locus. We briefly explain the general principle behind this construction. For details we refer the reader to \cite{Bialynicki:73}.
	
	Let $X$ be a smooth projective variety over the algebraically closed field $k$. Let $T = (\mathbb{G}_m)^n$ be an algebraic torus and assume it acts on $X$. The locus of $T$-fixed points $X^T$ is closed. For a one-parameter subgroup $\lambda$ of $T$ consider the induced $\mathbb{G}_m$-action on $X$. The set of fixed points $X^{\lambda(\G_m)}$ of this action contains $X^T$. It is always possible to choose $\lambda$ sufficiently generically such that $X^T = X^{\lambda(\G_m)}$. Let $x \in X(k)$, i.e.\ a closed point of $X$. Consider the morphism $\G_m \to X$ which on closed points is given by $z \mapsto \lambda(z)x$. This morphism can be extended uniquely to a morphism $\A^1 \to X$ as $X$ is separated and complete. We denote its value at $0$ by $\lim_{z \to 0} \lambda(z)x$. This point is automatically a $\lambda(\G_m)$-fixed point. Let $C_\nu$ be a connected component of the fixed point locus. Then by \cite[Thm.\ 4.1]{Bialynicki:73} there is a locally closed subvariety $A_\nu(\lambda)$, called the attractor of $C_\nu$, whose closed points are
	$$
		\{x \in X(k) \mid \lim_{z \to 0} \lambda(z)x \text{ lies in } C_\nu \}.
	$$ 
	Moreover the map $A_\nu(\lambda) \to C_\nu$ is an affine bundle (even a vector bundle according to \cite[II, Thm.\ 4.2]{BCM:02}). As $X$ is irreducible, there is precisely one attractor which is open, say $A_{\nu_0}(\lambda)$. If $C_{\nu_0}$ is rational then so is $A_{\nu_0}(\lambda)$ and this would imply rationality of $X$. Summarizing:
	
	\begin{lem} \label{l:BB}
		Let $X$ be a smooth projective (irreducible) variety which is acted upon by a torus $T$. Let $\lambda$ be a generic one parameter subgroup for this action and let $C_{\nu_0}$ be the unique component of the fixed point locus $X^T$ for which the attractor $A_{\nu_0}(\lambda)$ is open in $X$. If $C_{\nu_0}$ is rational then $X$ is rational, too.
	\end{lem}
	
	The first torus action that we are going to use reduces the case of a rigid representation to the case of an exceptional representation. It has already been introduced in the proof of \cite[Prop.\ 3.2]{DWZ:10}. Let $M = M^{(1)} \oplus \ldots \oplus M^{(n)}$ be the decomposition of a rigid representation $M$ into indecomposable direct summands. Then every $M^{(r)}$ is exceptional. Let $d^{(r)} = \dimvect M^{(r)}$. Consider the action of $T = (\G_m)^n$ on $\Gr_e(M)$ as follows: for $t = (t_1,\ldots,t_n)$ and $U \in \Gr_e(M)$ let $u = \sum_r u_r$ be the decomposition of $u \in U$ into summands $u_r \in \smash{M^{(r)}}$. Define $t*u = \sum_r t_ru_r$ and $t*U = \{t*u \mid u \in U\}$. The subspace $t*U$ is in fact a subrepresentation of $M$ of dimension vector $e$, hence it belongs to $\Gr_e(M)$. A subrepresentation $U$ of $M$ is a fixed point of this action if and only if $U = \bigoplus_r U \cap M^{(r)}$. We conclude that
	$$
		\Gr_e(M)^T = \bigsqcup_{e^{(r)} \leq d^{(r)},\ \sum_r e^{(r)} = e} \Gr_{e^{(1)}}(M^{(1)}) \times \ldots \times \Gr_{e^{(n)}}(M^{(n)}).
	$$
	An application of Lemma \ref{l:BB} now shows
	
	\begin{lem} \label{l:red1} 
		Let $M$ be a rigid representation of $Q$ and let $M = M^{(1)} \oplus \ldots \oplus M^{(n)}$ be its decomposition into indecomposable direct summands. If every quiver Grassmannian associated with each $M^{(r)}$ is rational, then $\Gr_e(M)$ is rational. 
	\end{lem}

	The second torus action reduces the case of an exceptional representation of an arbitrary quiver to the case of an exceptional representation of a tree-shaped quiver. A variant of this action was employed in \cite[4.2]{RW:18}.
	Let $M$ be an exceptional representation of $Q$. By Lemma \ref{l:ringel} $M$ lifts to a universal cover. So it also possesses a lift to the universal abelian covering quiver. Fix such a lift $\smash{\hat{M}}$, i.e.\ fix a decomposition $M_i = \bigoplus_\chi \smash{\hat{M}_{i,\chi}}$ such that $M_\alpha(\smash{\hat{M}_{s(\alpha)}}) \sub \smash{\hat{M}_{t(\alpha),\chi+\alpha}}$.
	Let $T$ be the algebraic torus $\smash{(\mathbb{G}_m)^{Q_1}}$. Its group of characters is precisely $\smash{\Z^{Q_1}}$. Let $G = \prod_{i \in Q_0} \GL(M_i)$. Let $\psi: T \to G$ be the group homomorphism that corresponds to the chosen decomposition. More precisely this means that $\psi$ is given by group homomorphisms $\psi_i: T \to \GL(M_i)$ which are defined as
	$$
		\psi_i(t): M_i \to M_i, v_i \mapsto \sum_\chi \chi(t)v_{i,\chi}
	$$ 
	where $v_i = \sum_\chi v_{i,\chi}$ is the decomposition of a vector with respect to the direct sum decomposition. We define a $T$-action on $\Gr_e(M)$ via $t * U = (\psi_i(t) U_i)_i$. To show that this is well-defined we need to make sure that $t * U$ is again a subrepresentation of $M$. Observe that for every arrow $\alpha: i \to j$ 
	$$
		\psi_j(t)M_\alpha \psi_i(t)^{-1} = t_\alpha M_\alpha.
	$$
	Hence we see that
	$$
		M_\alpha(\psi_i(t)U_i) = t_\alpha^{-1}\psi_j(t)M_\alpha\psi_i(t)(\psi_i(t)U_i) = t_\alpha^{-1}\psi_j(t)M_\alpha(U_i) \sub t_a^{-1}\psi_j(t)U_j = \psi_j(t)U_j
	$$
	which proves that, indeed, $t*U$ is a subrepresentation of $M$. A subrepresentation $U$ of $M$ is a fixed point of this torus action on $\Gr_e(M)$ if and only if
	$$
		U_i = \bigoplus_\chi U_i \cap \hat{M}_{i,\chi}.
	$$
	Let $\hat{d}$ be the dimension vector of $\smash{\hat{M}}$. The above argument shows that the fixed point locus $\Gr_e(M)^T$ decomposes as the disjoint union
	$$
		\bigsqcup_{\hat{e} \leq \hat{d},\ c(\hat{e}) = e} \Gr_{\hat{e}}(\hat{M})
	$$
	the components of which are quiver Grassmannians of $\smash{\hat{M}}$ which is a representation of $\hat{Q}$. 
	So, applying again Lemma \ref{l:BB}, we get
	
	\begin{lem} \label{l:red2}
		Let $M$ be an exceptional representation of $Q$ and let $\smash{\hat{M}}$ be a lift of $M$ to the abelian covering quiver. If every quiver Grassmannian associated with $\smash{\hat{M}}$ is rational, then so is $\Gr_e(M)$.
	\end{lem}

	\section{Rigid Quiver Grassmannians of Tree-Shaped Quivers}

	In this section let $Q$ be an orientation of a tree. Let $M$ be a rigid representation of $Q$. Let $l$ be a leaf of $Q$, i.e.\ a vertex with just one arrow adjacent to it. Without loss of generality we may assume $l$ to be a sink; if $l$ is a source we may argue in a similar way or apply the reflection functor at $l$ which defines a birational equivalence $\Gr_e(M) \dashrightarrow \Gr_{\sigma_l(e)}(S_l^+(M))$ (see \cite[Thm.\ 5.16]{Wolf:09}). 
	
	Let $\beta: k \to l$ be the unique arrow ending in $l$. 
	Let $Q'$ be the full subquiver of $Q$ on $Q'_0 = Q_0 - \{l\}$. For a representation $N$ of $Q$, define $N'$ as the representation of $Q'$ which arises by forgetting about $N_l$ and $N_\beta$. For a dimension vector $d$ of $Q$ let $d' = (d_i)_{i \neq l}$.
	
	\begin{lem} \label{l:rigid}
		Let $M$ be a rigid representation of $Q$. Then $M'$ is a rigid representation of $Q'$.
	\end{lem}
	
	\begin{proof}
		Consider the short exact sequence $0 \to S(l)\otimes M_l \to M \to M'' \to 0$ of representations of $Q$. It is obvious that $\Ext_{Q'}(M',M') = \Ext_Q(M'',M'')$. It hence suffices to show that $M''$ is a rigid representation of $Q$. An application of $\Hom_Q(M,\blank)$ (we are going to drop the subscript $Q$ in the following) to the short exact sequence above yields a surjection
		$$
			\ldots \to \Ext(M,M) \to \Ext(M,M'') \to 0
		$$
		and hence $\Ext(M,M'') = 0$ for $M$ is rigid. Applying $\Hom(\blank,M'')$ to the same short exact sequence renders
		$$
			\ldots \to \Hom(S(l) \otimes M_l,M'') \to \Ext(M'',M'') \to \Ext(M,M'') \to \ldots.
		$$
		The group $\Ext(M,M'')$ vanishes as we have just pointed out and $\Hom(S(l) \otimes M_l,M'')$ vanishes too, as $M''_l = 0$. This shows that $\Ext(M'',M'') = 0$.
	\end{proof}
	
	We consider the morphism $p_l: \Gr_e(M) \to \Gr_{e'}(M')$ of varieties which sends $U$ to $p_l(U) = U'$. For a subrepresentation $U'$ of $M'$ the set of closed points of the fiber $p_l^{-1}(U')$ is the set
	$$
		\{U_l \in \Gr_{e_l}(M_l) \mid M_\beta(U'_k) \sub U_l \}.
	$$
	Consider the map of vector bundles on $\Gr_{e'}(M')$
	$$
		\phi_\beta: \UU'_k \to (M'_k)_{\Gr_{e'}(M')} = (M_k)_{\Gr_{e'}(M')} \xto{}{M_\beta} (M_l)_{\Gr_{e'}(M')}
	$$
	Let $X \sub \Gr_{e'}(M')$ be the dense open subset over which $\phi_\beta$ has maximal rank, say $r$. Then the image $\im \phi_\beta$ is a vector bundle of rank $r$ on $X$ (note that $\Gr_{e'}(M')$ is smooth and hence reduced) and the restriction $p_l^{-1}(X) \to X$ of $p_l$ identifies the the total space $p_l^{-1}(X)$ with the subbundle Grassmannian $\Gr_{e_l-r}((M_l)_X/\im \phi_\beta)$ as a scheme over $X$.
	
	\begin{lem} \label{l:vb}
		Let $X$ be rational variety, let $E$ be a vector bundle on $X$ of rank $n$ and let $r \leq n$. Then $\Gr_r(E)$ is a rational variety.
	\end{lem}
	
	\begin{proof}
		Let $U$ be a non-empty open subset of $X$ over which $E$ trivializes, say $E|_U \cong U \times k^n$. Let $p: \Gr_r(E) \to X$ be the structure map. Then $p^{-1}(U) \cong U \times \Gr_r^n$. As both $U$ and $\Gr_r^n$ are rational, then so is $p^{-1}(U)$ and thus also $\Gr_r(E)$.
	\end{proof}
	
	The above lemma shows that, in order to show that $\Gr_e(M)$ is rational, it suffices to prove that $X$ is rational, or equivalently that $\Gr_{e'}(M')$ is rational.	
	
	\begin{lem} \label{l:treeRational}
		Let $Q$ be tree-shaped and let $M$ be a rigid representation of $Q$. Then $\Gr_e(M)$ is a rational variety.
	\end{lem}
	
	\begin{proof}
		We choose a leaf of $Q$, say $l$, of which we may without loss of generality assume that it is a sink. Over a suitable non-empty open subset $X \sub \Gr_{e'}(M')$ the map $p_l: \Gr_e(M) \to \Gr_{e'}(M')$ is a Grassmann bundle. Hence by Lemma \ref{l:vb}  $\Gr_e(M)$ is rational provided that $\Gr_{e'}(M')$ is. But $M'$ is also rigid by Lemma \ref{l:rigid} hence we may assume by induction on the number of vertices of the quiver that $\Gr_{e'}(M')$ is a rational variety. The base of this induction is the case of a single vertex and no arrows; in this case the quiver Grassmannian is an ordinary Grassmannian which is rational.
	\end{proof}

	\section{Proof of the Main Result}
	
	We are ready to prove Theorem \ref{t:main} which states that the quiver Grassmannian $\Gr_e(M)$ associated with a rigid representation $M$ is a rational variety.
		
	\begin{proof}[Proof of Theorem \ref{t:main}]
		Let $C$ be a universal cover of $Q$. Using Lemma \ref{l:red1}, we may assume without loss of generality that $M$ is exceptional. Being exceptional, $M$ possesses a lift $N$ to $C$ by Lemma \ref{l:ringel}. Choose cover morphisms $\smash{p^{(n)}}$ as in Lemma \ref{l:covermaps} and a natural number $n$ large enough such that $\smash{p^{(n)}(N)}$ is supported on a tree, which is possible according to Lemma \ref{l:lift}. Inductively applying Lemma \ref{l:red2}, we see that $\Gr_e(M)$ is rational provided that every quiver Grassmannian $\Gr_f(\smash{p^{(n)}(N)})$ is. But as $\smash{p^{(n)}(N)}$ is exceptional by virtue of Lemma \ref{l:rigidCover} and supported on a tree, rationality of the quiver Grasmmannian $\Gr_f(\smash{p^{(n)}(N)})$ follows from Lemma \ref{l:treeRational}. This concludes the proof.
	\end{proof}

	\bibliographystyle{abbrv}
	\bibliography{Literature}

\end{document}